\documentclass[12pt]{article}
\usepackage[margin=1.0in]{geometry}
\usepackage[colorinlistoftodos,bordercolor=orange,backgroundcolor=orange!20,linecolor=orange,textsize=scriptsize]{todonotes}

\usepackage{amssymb}
\usepackage{amsthm}
\usepackage{amsmath}
\usepackage{setspace}
\usepackage{natbib}
\usepackage{bbm}
\usepackage{graphicx}
\usepackage{tikz}
\usepackage{tikz-cd}
\usepackage{mathrsfs}
\usepackage{enumerate}

\usepackage{hyperref}

\newtheorem{theorem}{Theorem}[section]
\newtheorem{definition}[theorem]{Definition}
\newtheorem{lemma}[theorem]{Lemma}

\newtheorem{corollary}[theorem]{Corollary}

\newtheorem{remark}[theorem]{Remark}
\newtheorem{question}[theorem]{Question}
\newtheorem*{theorem*}{Theorem}

\newcommand{\RNum}[1]{\lowercase\expandafter{\romannumeral #1\relax}}

\title{Linear quotients, linear resolutions and the lcm-lattice}
\author{Roni Varshavsky}
\date{}

\begin{document}

\newpage
\maketitle
\begin{abstract}
    Linear resolutions and the stronger notion of linear quotients are important properties of monomial ideals. In this paper, we fully characterize linear quotients in terms of the lcm-lattice of monomial ideals. We also formulate an analogous characterization for monomial ideals with linear resolutions, making explicit a relationship that is implicit in the existing literature.
    These results complement characterizations of these two properties in terms of the Alexander dual of the corresponding Stanley-Reisner simplicial complex.
    In addition, we discuss applications to the case of edge ideals.
\end{abstract}
\tableofcontents
\newpage

\newpage

\section{Introduction}

Finitely generated graded modules over finitely generated graded algebras are central objects in commutative algebra, and a variety of algebraic invariants have been introduced to better understand their structure. In particular, the study of free resolutions of such modules has provided deep insights into their algebraic properties. For a finitely generated graded module $U$ over a finitely generated graded algebra $R$, there exists a unique minimal graded free resolution \cite{peeva2010graded} (up to graded isomorphism), and the associated graded Betti numbers are important invariants describing this resolution.

A fundamental and widely studied class of graded modules is that of monomial ideals. These ideals form a natural bridge between commutative algebra and combinatorics, since they give rise to combinatorial structures whose properties reflect algebraic properties of the ideal. The connections between such structures and the algebraic behavior of the corresponding ideals have been explored extensively, for instance through Hochster’s formula (\cite{bruns1998cohen}, 5.5),
see e.g. \cite{RalfFröberg1990, herzog2002resolutions, gasharov1999lcm, Nevo2013C4freeEI}.

Among the various algebraic properties of monomial ideals, the notion of linear quotients plays an important role. Introduced by Herzog and Takayama in \cite{herzog2002resolutions}, linear quotients provide a stronger condition than having a linear resolution, and they have proven to be a powerful combinatorial tool. In particular, they showed that a Stanley–Reisner ideal has linear quotients if and only if its Alexander dual is shellable in the non-pure sense of Björner and Wachs \cite{Bjrner1996ShellableNC}.

Building on this connection between algebraic and combinatorial structures, we focus on characterizing linear quotients through the lcm-lattice of a monomial ideal. The lcm-lattice, introduced by Gasharov, Peeva, and Welker \cite{gasharov1999lcm}, encodes the divisibility relations among least common multiples of the minimal generators of a monomial ideal and has been shown to reflect many of its homological invariants. Phan (\cite{phan2005minimal}, Thm. 2.1) further established that every finite atomic lattice arises as the lcm-lattice of a monomial ideal, 
thereby emphasizing the central role of this structure in combinatorial commutative algebra.

In \cite{bjorner1983lexicographically}, Bj\"orner and Wachs introduce stronger notions of shellability for order complexes of posets, namely \textit{edge-lexicographical shellability} (EL-shellability) and \textit{chain-lexicographical shellability} (CL-shellability). In this paper, we provide a new characterization of monomial ideals with linear quotients in terms of their lcm-lattices. Specifically, we prove the following main result (see \ref{degree_graded_definition} for definition of degree-gradedness):
\begin{theorem}
    \label{linear_quotients_cl_shellability}
    Let $I$ be a monomial ideal over some field $k$ generated in degree $d$, $L=L(I)$ its lcm-lattice. Then $I$ has linear quotients if and only if $L$ is d-degree graded and $CL$-shellable. 
\end{theorem}
In view of [\cite{herzog2011monomial}, Proposition 8.2.5] the CL-shellability condition in Theorem \ref{linear_quotients_cl_shellability}, rather than just shellability, is a bit mysterious. See further discussion in Question \ref{cl_shellability_and_shellability_question}.

In the one dimensional case, it is well known that 
an edge 
ideal has linear quotients if and only if the underlying graph is cochordal ([\cite{RalfFröberg1990}, Theorem 1] + [\cite{herzog2003monomialidealspowerslinear}, Theorem 3.2]). 
The following corollary follows.
\begin{corollary}
\label{chordality_cl_shellability}
    Let $G = (V,E)$ be a simple graph, $k$ be a field, $I$
    be its edge ideal over $k$ and $L = L(I)$ its lcm-lattice. Then $G$ is cochordal if and only if $L$ is 2-degree graded and CL-shellable.
\end{corollary}
For completeness, we state another theorem, which 
is implicit in
the literature, 
in a formulation that complements our main result. 
It characterizes the monomial ideals with linear resolutions in terms of their lcm-lattice, proved explicitly in the Appendix:

\begin{theorem}
    \label{linear_resolutions_cm_lcm_lattice}
    Let $I$ be a monomial ideal over some field $k$, $L=L(I)$ its lcm-lattice. Then $I$ has a $d$-linear resolution if and only if $L(I)$ is $d$-degree graded and Cohen-Macaulay.
\end{theorem}
\section{Preliminaries}

\subsection{Monomial ideals}

For an introduction to linear resolutions and graded Betti numbers, we refer the reader to \cite{peeva2010graded}.

\begin{lemma}
Let $I\triangleleft S$  be a monomial ideal. Then there exists a unique minimal set of monomial generators of $I$, denoted by $G(I)$.
\end{lemma}

Let $R$  be any commutative ring, $I,J\triangleleft R$.  The colon ideal $(I:J)$ is $$(I:J)=\{r\in R\mid rJ\subseteq I\}.$$ 
\begin{definition}
    Let $I$  be a monomial ideal. We say that $I$  has linear quotients if there exists an ordering on the generators $G(I)$, namely $m_1,\dots,m_k$  such that $\forall 1\leq i\leq k$ the colon ideal $$((m_1,\dots,m_{i-1}):m_i)$$ is generated by variables.
\end{definition}

A useful fact in the case of monomial ideals is

\begin{lemma}
    \label{colon_structure}
    Let $I\triangleleft R$  be a monomial ideal, $v$ a monomial. Then $(I:(v))$  is a monomial ideal and $$\left\{\frac{u}{gcd(u,v)}\mid u\in G(I)\right\}$$ is a set of generators for $(I:(v))$.
\end{lemma}

\begin{theorem}[\cite{herzog2011monomial}, Theorem 8.2.1]
\label{linear_quotients_stronger_than_linear_resolution}
Let $I$  be a monomial ideal generated in degree $d$  (i.e. the elements of $G(I)$ are all monomials of degree $d$). 
If $I$ has linear quotients
then $I$  has a $d$-linear resolution.
\end{theorem}

The following definition is a useful technique for reducing problems on monomial ideals to the squarefree case:
\begin{definition}[Polarization]
\label{polarization_def}
    Let $I\subseteq S=k[x_1,\dots,x_n]$  be  a monomial ideal with minimal set of generators $G(I)=\{u_1,\dots,u_m\}$ with $u_i=\prod_{j=1}^n{x_j}^{a_{i,j}}$  $\forall1\leq i\leq m$. Denote for all $1\leq j\leq n$ $a_j=\text{max}_i\{a_{i,j}\}$ and consider the polynomial ring $$T=k[x_{11},x_{12},\dots,x_{1a_1},x_{21},\dots,x_{2a_2} ,\dots ,x_{n1},\dots ,x_{na_n}]$$ Denote now $\textit{for all } 1\leq i\leq m$ $$  v_i=\prod_{j=1}^{n}\prod_{k=1}^{a_{i,j}}x_{jk}$$ The polarization of $I$  is the ideal $J\triangleleft T$ generated by the elements $\{v_1,\dots ,v_m\}$. A standard notation for $J$ is $\tilde{I}$.
\end{definition}

For further reading about monomial ideals, see \cite{herzog2011monomial}.

\subsection{Simplicial complexes corresponding to monomial ideals}
Let $n\in\mathbb{N}$. A \textbf{simplicial complex} on vertices $[n]$ is a downward closed subset $\Delta$ of the powerset $\mathcal{P}([n])$. Every element $E\in \Delta$ is called a \textbf{face}.

 A \textbf{facet} of $\Delta$  is a maximal face $F\in \Delta$ w.r.t. inclusion and the set of facets is denoted by $\mathcal{F}(\Delta)$. A minimal non-face is a set $F\subseteq [n]$ such that $F\notin \Delta$  and $\forall T\subsetneq F,\ T\in \Delta$. A simplicial complex is determined by its set of facets, and given a set $S\subseteq \mathcal{P}([n])$ such that no two elements of $S$  are contained in each other, we denote by $\left<S\right>$  the simplicial complex $\Delta$  with $\mathcal{F}(\Delta)=S$.
 The \textbf{dimension}  of a face $F\in \Delta$ is equal to $|F|-1$, and the dimension of a simplicial complex is $$\text{dim}(\Delta)=\text{max}_{F\in \Delta}(\text{dim}(F)).$$ The \textbf{link}  of a face $F$ in $\Delta$ is the set $\left\{T\in \Delta: \ T\cup F\in \Delta, T\cap F = \emptyset\right\}$  and is denoted $\text{link}_{\Delta}F$.

\begin{definition}
    Let $n\in \mathbb{N}$,  $S=k[x_1,\dots ,x_n]$ and $F\subseteq [n]$.  We write $x^F=\prod_{i\in F}x_i$. Now let $\Delta$ be a simplicial complex, the \textbf{Stanley-Reisner ideal} of $\Delta$  is $I_{\Delta} =\left(x^F, F\notin \Delta\right)$. It is generated by the monomials corresponding to the \textbf{minimal non-faces} of $\Delta$. The \textbf{facet ideal} is the monomial ideal generated by the facets of $\Delta$. Given an ideal $I\triangleleft S$  we denote by $\Delta(I)$ the simplicial complex $\left\{F\subseteq [n], x^F\notin I\right\}$ 
\end{definition}

It is easy to see the following:

\begin{lemma}
    \label{stanley-reisner correspondence}
    For a squarefree monomial ideal $I$, we have $I_{\Delta(I)}=I$ and for a simplicial complex $\Delta$, we have $\Delta(I_{\Delta})=\Delta$ 
\end{lemma}
For a simplicial complex $\Delta$, the \textbf{Alexander dual} is denoted by $\Delta^{\vee}$  and is equal to $\Delta ^{\vee}=\left\{[n]\setminus F,F\notin \Delta\right\}$. It is straightforward that the Alexander dual is also a simplicial complex satisfying $$\Delta^{\vee}=\left<[n]\setminus F, \text{F is a minimal non-face of } \Delta\right>.$$

One important combinatorial property of simplicial complexes is \textbf{shellability}. Topologically, shellable complexes are homotopy equivalent to a wedge sum of spheres. However the converse is not true, as there are examples of non-shellable triangulations of spheres e.g. \cite{VINCE198591}.
\begin{definition}
    A simplicial complex $\Delta$ is \textbf{pure} if all of its facets are of the same dimension. A pure simplicial complex is called \textbf{shellable} if there exists an ordering of its facets $F_1,\dots , F_k$  such that for all $1\leq i\leq k$ the intersection $$2^{F_i} \bigcap \left(\bigcup _{j=1} ^{i-1}2^{F_j}\right)$$ is pure of dimension $dim(F_i)-1$.
\end{definition}

    A similar in spirit, weaker property for pure simplicial complexes is \textbf{constructibility}. It is defined recursively as follows:
    \begin{definition}
    A pure simplicial complex $\Delta$ is \textbf{constructible} if either
    \begin{enumerate}
        \item $\Delta$ is a simplex, or
        \item $\Delta=\Delta_1\cup \Delta_2$ such that $\text{dim}(\Delta_1)=\text{dim}(\Delta_2)=\text{dim}(\Delta_1\cap\Delta_2)+1$ and $\Delta, \Delta_1, \Delta_2$  are all constructible.
    \end{enumerate}
    \end{definition}

Another property, topological this time, is \textbf{Cohen-Macaulayness}:
\begin{definition}
    A simplicial complex $\Delta$ is \textbf{Cohen-Macaulay} over a field $k$ if $$\textit{for all } F\in \Delta \ \tilde{H}_i(\text{link}_{\Delta}F;k)=0 \ \textit{for all } i<\text{dim}(\text{link}_{\Delta}F)$$
\end{definition}

\subsection{The lcm-lattice}

First, some basic definitions regarding posets:

\begin{definition}
    Let $P$ be a poset. For $a,b\in P$ we say that $b$ \textbf{covers}  $a$  if $a<b$ and $\forall a\leq c \leq b $  we have either $b=c$  or $a=c$, and we write $a\prec b$. Such a pair $(a,b)$  is called an \textbf{edge} in $P$. If $P$  has a minimal element $\hat{0}$  then elements covering it are called \textbf{atoms}. A finite lattice where every element is a join of atoms is called \textbf{atomic}. 
\end{definition}

For every finite poset there is a natural simplicial complex associated to it:
\begin{definition}
    Let $P$  be a finite poset. A chain in $P$  is a totally ordered subset. A chain in $P$  is called \textbf{unrefinable} (or saturated) if it is of the form $(a_1\prec a_2 \prec \dots \prec a_n)$. The order complex of $P$, denoted by $O(P)$ is the simplicial complex with vertices the elements of $P$  and faces the chains of $P$.  A finite poset is said to be \textbf{graded} if it is bounded and the order complex $O(P)$  is pure.
\end{definition}

\begin{definition}
    Let $I$  be a monomial ideal with minimal generating set $G(I)=\{m_1,\dots,m_r\}$. The \textbf{lcm-lattice} of $I$,  denoted by $L(I)$ or $LCM(I)$ is the finite atomic lattice consisting of the least common multiples of all subsets of $G(I)$, ordered by divisibility. $1$  is the unique minimal element of the lattice, representing the lcm of the empty set. For an lcm-lattice $L$, we denote by $m_L$ the unique maximal element.
\end{definition}

The following equivalence between finite atomic lattices and lcm-lattices was proven by Phan in \cite{phan2005minimal}.

\begin{theorem}
\label{minimal_ideal_definition}
    Let $L$  be a finite atomic lattice. Then there exists a monomial ideal $I\triangleleft k[x_1,\dots,x_n]$  for some $n$  such that $L\cong L(I)$. This ideal is called the minimal ideal of $L$ and is denoted by $M(L)$.
\end{theorem}

The following theorem connects the topology of the lcm-lattice of a monomial ideal $I$  and the graded Betti number $\beta_{i,j}(I)$. It was proven by Gasharov, Peeva and Welker in \cite{gasharov1999lcm} (see also \cite{peeva2010graded}, Theorem 58.8):

\begin{theorem}
    \label{betti_lcm_lattice_connection}
    Let $I\triangleleft k[x_1,\dots ,x_n]$ be a monomial ideal, $L=L(I)$  its lcm-lattice. The multigraded Betti numbers of $I$ are given by $$\beta_{i,m}(I)=\begin{cases}
        \textrm{dim} \ \tilde{H}_{i-1, m}(O((1,m));k) & 1\neq m \in L \\
        0 & m \notin L
    \end{cases}$$
\end{theorem}
As for the lcm-lattice of a polarization of a monomial ideal, we have
\begin{lemma}[\cite{peeva2010graded}, 58.4]
    \label{polarization_lattices_iso}
    Let $I$  be a monomial ideal, $J$ its polarization, $L(I), L(J)$ the respective lcm-lattices. Then $L(I)\cong L(J)$.
\end{lemma}

Bj\"orner and Wachs \cite{bjorner1983lexicographically} presented a notion of \textit{CL-shellability} of graded posets that is stronger than shellability of posets. They also introduced the equivalent notion of Recursive atom orderings of posets:

\begin{definition}
\label{rao_definition}
    Let $P$  be a graded poset. $P$  admits a \textbf{recursive atom ordering} (RAO) if either $\bar{P}=\emptyset$ (where $\bar{P}=P-\{\hat{0}, \hat{1}\})$ 
    or if there is an ordering of the atoms $a_1,a_2,\dots ,a_t$ of $P$  such that:
    \begin{enumerate}[(i)]
\item $\textrm{for all } 1\leq j\leq t$ the interval $[a_j,\hat{1}]$  admits a recursive atom ordering such that the atoms of $[a_j,\hat{1}]$  that cover some $a_i$  for some $i<j$  appear before the atoms that do not have this property.
        \item $\textrm{for all } 1\leq i<j\leq t $ there exists $k<j$  such that $\textrm{for all } y, \ a_i,a_j\leq y$ there exists $z\in P$  such that $a_k,a_j \prec z \leq y$.
    \end{enumerate}
\end{definition}

We recall the definitions of semimodularity and total semimodularity:

\begin{definition}
\label{semimodularity}
A poset $P$  is called \textbf{upper semimodular}, or sometimes simply \textbf{semimodular} if $\textrm{for all } a,b \in P$ such that $c\prec a,b$  for some $c\in P$, then there exists $d\in P$ such that $a,b \prec d$.
    Reducing to lattices, a finite lattice $L$  is called \textbf{semimodular} if $\forall a,b\in L\ s.t. \ a\wedge b \prec a,b$  it holds that $a,b\prec a\vee b$. A finite poset $P$  is said to be \textbf{totally semimodular}  if it is bounded and for every closed interval $[x,y]\subseteq P$, $[x,y]$ is semimodular.
\end{definition}

\section{The lcm-lattice of ideals with linear resolutions and linear quotients}

\label{linear_quotients_ideals_cl_shellability_section}
Having linear resolutions or linear quotients are important properties for a monomial ideal. It is therefore interesting to study the behavior of associated combinatorial structures in this case.

\subsection{Shellability and linear quotients}
For linear quotients, there is a correspondence in the language of Alexander duality introduced in \cite{herzog2002resolutions} and proven in \cite{herzog2003monomialidealspowerslinear}, namely:

\begin{theorem}[\cite{herzog2011monomial}, Proposition 8.2.5]
\label{linear_quotients_alexander_dual}
    Let $\Delta$ be a simplicial complex. Then $I_\Delta$ has linear quotients if and only if $\Delta ^{\vee}$ is shellable.
\end{theorem}
By Lemma \ref{stanley-reisner correspondence}, every squarefree monomial ideal is of the form $I_\Delta$ for some simplicial complex $\Delta$, hence the above theorem is a characterization for all squarefree monomial ideals with linear quotients.

\begin{remark}
\label{linear_quotients_alexander_dual_with_degrees}
    We can give a precise formulation with degrees in the pure case. Namely:
    Let $\Delta$ be a simplicial complex on $n$ vertices. Then $I_\Delta$ is generated in degree $q$ and has linear quotients if and only if $\Delta ^{\vee}$ is pure of dimension $n-q-1$ and shellable.
\end{remark}

The characterization we give in terms of the lcm-lattice will be true for monomial ideals whose generators are all of the same degree, but can be formulated without the assumption of them being squarefree (and without requiring polarization in the definition), which is somewhat natural. 
We start by introducing a notion that will be useful to us:

\begin{definition}
\label{degree_graded_definition}
    We say that an lcm-lattice $L(I)$ is \textbf{degree-graded}  if $\textrm{for all } 1\neq a,b\in L, \ a\prec b \Rightarrow deg(b)=deg(a)+1$ . We say that it is \textbf{d-degree graded} if it is degree graded and the atoms of $L$ are all of degree $d$ (as monomials).
\end{definition}

Bj\"orner and Wachs showed the following:

\begin{theorem}[\cite{bjorner1983lexicographically}, Theorem 5.1]
    \label{total_semimodularity_rao}
    A graded poset $P$  is totally semimodular if and only if for every interval $[x,y]\subseteq P$, every ordering of its atoms is a recursive atom ordering.
\end{theorem}

Now a lemma regarding semimodularity in degree graded lcm-lattices:

\begin{lemma}
    \label{d_degree_graded_gives_total_semimodularity}
    Let $I$  be a squarefree monomial ideal such that $L=L(I)$  is $d$-degree graded. Then every interval $[x,y]$, $x\neq 1$ is semimodular. In particular, every such interval is totally semimodular.
\end{lemma}
\begin{proof}
    Let $a,b \in [x,y]$ such that $a\wedge b \prec a,b$. Since $L$ is $d$-degree graded, we have $deg(a \wedge b)=deg(a)-1=deg(b)-1$. This means that $a,b$ are of the form $a=(a \wedge b)x_a, \ b=(a\wedge b)x_b$ for some variables $x_a\neq x_b$. Thus $a\vee b = (a\wedge b)x_ax_b$, giving that $a,b\prec a\vee b$ by degree considerations. Clearly also $a\vee b \le y$.
\end{proof}

A corollary from Theorem \ref{total_semimodularity_rao} and Lemma \ref{d_degree_graded_gives_total_semimodularity} would then be:
\begin{corollary}
\label{only_condition_2}
Let $I$  be a squarefree monomial ideal such that $L=L(I)$ is $d$-degree graded. Then if $a_1,\dots ,a_t$ is some ordering of the atoms of $L$  such that condition (\RNum{2}) of \ref{rao_definition} holds, then it is a RAO.
\end{corollary}
\begin{proof}
By Lemma \ref{d_degree_graded_gives_total_semimodularity}, every interval of the form $[a_j ,m_L]$ is totally semimodular $\textrm{for all } 1\leq j\leq t$. Thus by Theorem \ref{total_semimodularity_rao} every ordering of the atoms of $[a_j,m_L]$  is a RAO. Thus condition (\RNum{1}) of \ref{rao_definition} holds for any ordering of the atoms of $L$ (as in particular we can choose one where the atoms that come first are those that cover some $a_i, \ i<j$). Since condition (\RNum{2}) holds by assumption for the ordering at hand, we have that it is a RAO as well.
\end{proof}

Another convenient terminology will be the following:
\begin{definition}
    Let $I\triangleleft k[x_1, \dots ,x_n]$  be a monomial ideal. A \textbf{linear divisor} of $I$ is a monomial $m$ such that the colon ideal $(I:m)$ is generated by variables.
\end{definition}

We now proceed to proving our main result. First, in the squarefree case.
\begin{theorem}
    \label{linear_quotients_cl_shellability_squarefree}
    Let $I\triangleleft k[x_1, \dots ,x_n]$ be a squarefree monomial ideal generated by monomials of degree $d$, $L=L(I)$  its lcm-lattice. The following are equivalent:
    \begin{enumerate}
        \item $L$  is $d$-degree graded and CL-shellable.
        \item $I$ has linear quotients.
    \end{enumerate}
\end{theorem}
\begin{proof}
    Suppose that I has linear quotients. Let $a_{1},...,a_{t}$ be an ordering of its generators such that $\forall 1\leq i \leq t-1$, $a_{i+1}$ is a linear divisor of $\left(a_{1},...,a_{i}\right)$. To show that $L$ is $d$-degree graded one can use Lemma \ref{linear resolution implies degree graded}  and 
 Theorem \ref{linear_quotients_stronger_than_linear_resolution}, but let us show a more direct proof without going through linear resolutions:

Let $m_{1}<m_{2}$ in $L$ such that $deg\left(m_{2}\right)-deg\left(m_{1}\right)\geq2$. We need to show the existence of an element $m\in L$ s.t. $m_{1}<m<m_{2}$ in $L$. Let $m_{1}=a_{i_{1}}\vee...\vee a_{i_{c}}$ be an exhaustive (i.e., including all atoms below $m_1$) representation of $m_{1}$ as a join of atoms such that $1\leq i_{1}<...<i_{c}\leq t$. An exhaustive representation as a join of atoms for $m_{2}$ would thus be $m_{2}=a_{i_{1}}\vee...\vee a_{i_{c}}\vee a_{j_{1}}\vee...\vee a_{j_{b}}$ 
since $m_{1}\mid m_{2}$. Note that $b>0$ because $m_{1}$ and $m_{2}$ are not equal, and since the representation for $m_{1}$ was exhaustive we have that $a_{j_{l}}$ is divisible by a variable that none of the $a_{i_{q}}$ are divisible by $\forall1\leq l\leq b,1\leq q\leq c$ or else $a_{j_l}$ would already appear in the representation of $m_{1}$. W.l.o.g we assume $j_{1}<...<j_{b}$. 

\textbf{Case $i_{1}<j_{1}$}. Since $a_{j_{1}}$ is a linear divisor of $\left(a_{1},...,a_{j_{1}-1}\right), \exists k<j_{1}$ such that 
$$x_{q}=\frac{a_{k}}{gcd\left(a_{k},a_{j_{1}}\right)}\mid \frac{a_{i_{1}}}{gcd\left(a_{i_{1}},a_{j_{1}}\right)}$$ for some variable $x_{q}$. This gives that $$a_{k}=\frac{a_{j_{1}}x_{q}}{x_{p}}$$ for some variable $x_{p}$. We now notice that if $deg\left(m_{1}\vee a_{j_{1}}\right)>deg\left(m_{1}\right)+1, a_{j_{1}}$ is divisible by at least two variables such that none of the $a_{i_{l}}$'s are divisible by them. This means that $a_{k}$ is divisible by a variable that $m_{1}$ is not divisible by since we divided out at most one such variable. Note that $a_{k}\mid m_{2}$, because both $a_{j_{1}}$ and 
$x_{q}\mid a_{i_{1}}$ do. Thus we get that $a_{k}$ is among the $a_{j_{l}}$'s in that case. But that is impossible because $k<j_{1}$ and $j_{1}$ was the minimal index among these elements. We conclude that $deg\left(m_{1}\vee a_{j_{1}}\right)=deg\left(m_{1}\right)+1$, hence $m_{1}<m_{1}\vee a_{j_{1}}<m_{2}$ (the second inequality is strict because of the assumption that $deg(m_2)>deg(m_1)+1$).

\textbf{Case $j_{1}<i_{1}$}. Similarly we get that $\exists k<i_{1}$ such that $$x_{q}=\frac{a_{k}}{gcd\left(a_{k},a_{i_{1}}\right)}\mid \frac{a_{j_{1}}}{gcd\left(a_{j_{1}},a_{i_{1}}\right)}$$ so $$a_{k}=\frac{a_{i_{1}}x_{q}}{x_{p}}$$ If $x_{q}\mid m_{1}$ then $a_{k}\mid m_{1}$ hence $a_k$ is among the $a_{i_{l}}$'s, contradicting minimality of $i_{1}$. Thus $m_{1}<m_{1}\vee x_{q}<m_{2}$ where the second inequality is (a) True because $a_{k}\mid m_{2}$ and (b) strict by degree considerations, by similar reasoning to the arguments in the previous case. This proves that $L$ is $d$-degree graded.

We now prove that $L$ is CL-shellable by constructing a RAO. By Corollary \ref{only_condition_2} it suffices to show that for some ordering condition (\RNum{2}) holds. 

Indeed, order the atoms in the same order as the linear quotients ordering, and let $i<j$. Like before, $\exists k<j$ such that $$a_{k}=\frac{a_{j}x_{q}}{x_{p}}$$ and $x_{q}\mid a_{i}$ while $x_{q}\nmid a_{j}$. Thus $a_{k}\mid (a_{i}\vee a_{j})$, hence $$a_{k},a_{j}\prec a_{k}\vee a_{j}\leq a_{i}\vee a_{j}$$ as desired (the cover relation is by degree considerations). 

Conversely, suppose that $L$ is d-degree graded and CL-shellable. We show that an ordering that gives a RAO of the atoms of $L$ is also a linear quotients ordering.

Indeed, let $a_{1},...,a_{t}$ be a RAO for $L$, let $2\leq j\leq t$ and $i<j$. By condition (\RNum{2}) we have that $\exists k<j$ such that $a_{k},a_{j}\prec a_{k}\vee a_{j}\leq a_{i}\vee a_{j}$. Since $L$ is $d$-degree graded, we have that $deg\left(a_{k}\vee a_{j}\right)=d+1$. Thus $$\frac{a_{k}}{gcd\left(a_{k},a_{j}\right)}=x_{q}$$ for some variable $x_{q}$. But since $a_{k}\mid (a_{i}\vee a_{j})$ and $x_{q}\nmid a_{j}$ we must have $x_{q}\mid a_{i}$. Thus $$x_{q}\mid\frac{a_{i}}{gcd\left(a_{i},a_{j}\right)}$$ which shows precisely that $a_{j}$ is a linear divisor of $\left(a_{1},...,a_{j-1}\right)$ and we are done.

\end{proof}
Theorem~\ref{linear_quotients_cl_shellability_squarefree} extends to monomial ideals that are not necessarily squarefree. We first provide a proof for a probably well known fact about polarization, but I haven't managed to find a reference for it:

\begin{lemma}
    \label{linear_quotients_polarization}
    Let $I$  be a monomial ideal, $J$  its polarization. Then $I$  has linear quotients if and only if $J$ has linear quotients.
\end{lemma}
\begin{proof}
Preserving the notations from \ref{polarization_def}, assume $I$  has linear quotients ordering and let $u_1,\dots ,u_m$  be an appropriate ordering of its generators. We show that the corresponding ordering $v_1,\dots ,v_m$  is also a linear quotients ordering by showing $v_i$ is a linear divisor of $(v_1,\dots ,v_{i-1}) \ \forall 2\leq i\leq m$.

For a monomial $t\in S$, denote by $deg_{x_i}(t)$  the largest integer $p$ such that $x_i^p\mid t$. In our case we have $deg_{x_q}(u_i)=a_{i,q} \forall i,q$. We will sometimes use $a_{i,q}$ and $deg_{x_q}(u_i)$ interchangeably in this proof.

Let $2\leq i \leq m$. 
From Lemma \ref{colon_structure} we know that 
$$((u_1,...,u_{i-1}):u_i)=\left(\left\{\frac{u_j}{gcd(u_j,u_i)}\right\}_{j=1}^{i-1}\right),$$ 
and by the definition of the ordering this ideal is generated by variables. So $\forall j<i$  $\exists k<i$  such that $\frac{u_k}{gcd(u_i,u_k)}=x_l$  for some $1\leq l\leq n$  that satisfies $x_l \mid \frac{u_j}{gcd(u_j,u_i)}$. Hence, $deg_{x_l}(u_k)=deg_{x_l}(u_i)+1$. 
For all $p\neq l$  $x_p\nmid \frac{u_k}{gcd(u_k,u_i)}$, so $\frac{v_k}{gcd(v_k,v_i)}=x_{l(a_{i,l}+1)}$. Because $x_l \mid \frac{u_j}{gcd(u_j,u_i)}$ we get $a_{j,l}\geq a_{i,l}+1=a_{k,l}$, hence 
$$\frac{v_k}{gcd(v_k,v_i)}=x_{l(a_{i,l}+1)}\mid \frac{v_j}{gcd(v_j,v_i)}$$
which means that $((v_1,...,v_{i-1}):v_i)$ is generated by variables as well, hence a linear quotients ordering.

Conversely, assume that $\{v_1,\dots,v_m\}$ is a linear quotients ordering. Let $1<i\leq m$ as before, then $((v_1,...,v_{i-1}):v_i)$  is generated by variables and we show that $((u_1,...,u_{i-1}):u_i)$  is as well, finishing the proof. Indeed, taking $j<i$  there exists $k<i$  such that 
$$\frac{v_k}{gcd(v_k,v_i)}=x_{l(a_{i,l}+1)}\mid \frac{v_i}{gcd(v_j,v_i)}$$ 
for some $1\leq l\leq n$ . Hence $deg_{x_q}(u_k)\leq deg_{x_q}(u_i)$  $\forall q\neq l$  and $deg_{x_l}(u_k)= deg_{x_l}(u_i)+1$. By the same reasoning, $deg_{x_l}(u_j)>deg_{x_l}(u_i)$ which gives us 
$$\frac{u_k}{gcd(u_k,u_i)}=x_{l}\mid \frac{u_i}{gcd(u_j,u_i)}$$ as desired.

\end{proof}

An easy lemma for lcm-lattices is:
\begin{lemma}
    \label{polarization_lattices_cl_shellability}
    Let $I$ be a monomial ideal, $J$  its polarization, $L(I)
    $ and $L(J)$ the corresponding lcm-lattices. We then have:
    \begin{enumerate}
        \item $L(I)$ is $d$-degree graded iff $L(J)$ is $d$-degree graded.
        \item $L(I)$ is CL-shellable iff $L(J)$ is CL-shellable
    \end{enumerate}
\end{lemma}
\begin{proof}
    Any isomorphism of posets between $L(I)$ and $L(J)$ gives equivalence 2. the isomorphism from Lemma \ref{polarization_lattices_iso} preserves degrees, ensuring equivalence 1.
\end{proof}
We can now prove Theorem \ref{linear_quotients_cl_shellability}:

\begin{proof}
    This follows from Theorem \ref{linear_quotients_cl_shellability_squarefree}, Lemma \ref{linear_quotients_polarization} and Lemma \ref{polarization_lattices_cl_shellability}:
    
\begin{tikzcd}
I\text{ has linear quotients} \arrow[d, Leftrightarrow, "\ref{linear_quotients_polarization}"] & & L(I)\text{ is }d-\text{degree graded and CL-shellable} \arrow[d, Leftrightarrow, "\ref{polarization_lattices_cl_shellability}"] \\
J\text{ has linear quotients} \arrow[rr, Leftrightarrow, "\ref{linear_quotients_cl_shellability_squarefree}"] & & L(J)\text{ is }d-\text{degree graded and CL-shellable}
\end{tikzcd}

\end{proof}

\subsection{Cohen-Macaulayness and linear resolutions}
For linear resolutions, an important correspondence is the Eagon-Reiner theorem (\cite{EAGON1998265}, Theorem 3):

\begin{theorem}[Eagon-Reiner Theorem]
\label{linear_resolutions_cm_alexander_dual}
Let $k$  be a field, $\Delta$ a simplicial complex with $n$  vertices, $I_\Delta$ its Stanley-Reisner ideal. Then $I_\Delta$  has a linear resolution over $k$  if and only if $\Delta^{\vee}$  is Cohen-Macaulay over $k$.
\end{theorem}

\begin{remark}
\label{eagon_reiner_theorem_with_degrees}
    We can give a precise formulation with degrees in the pure case. Namely:
    Let $k$ be a field, $\Delta$ be a simplicial complex on $n$ vertices, $I_\Delta$ its Stanley-Reisner ideal. Then $I_\Delta$ has a $q$-linear resolution over $k$ if and only if $\Delta ^{\vee}$ is pure of dimension $n-q-1$ and Cohen-Macaulay.
\end{remark}

For completeness of this work, we stated in Theorem \ref{linear_resolutions_cm_lcm_lattice} a characterization in terms of the lcm-lattice and in the language of our main result \ref{linear_quotients_cl_shellability}, which is implicit in current literature. A slightly less general statement was given in \cite{phan2005minimal}. In Appendix \ref{appendix:A} we give a detailed proof, which is a reformulation and combination of the work of Phan.

\subsection{Connections between the lcm-lattice and the Alexander dual}

We now observe the connections we have between the lcm-lattice and the Alexander dual. 
\begin{corollary}
\label{alexander_dual_lcm_lattice}
    Let $k$ be a field, $\Delta$ be a simplicial complex on $n$  vertices, $I_\Delta$ its Stanley-Reisner ideal and $L(I_\Delta)$ its lcm-lattice. Then the following holds:
    \begin{enumerate}
        \item $L(I_\Delta)$ is $d$-degree graded and Cohen-Macaulay over $k$ if and only if $\Delta^{\vee}$ is pure of dimension $n-d-1$  and Cohen-Macaulay over $k$.
        \item $L(I_\Delta)$ is $d$-degree graded and CL-shellable if and only if $\Delta^{\vee}$ is pure of dimension $n-d-1$ and shellable.
    \end{enumerate}
\end{corollary}
\begin{proof}

    The first part of this theorem comes from Theorems \ref{linear_resolutions_cm_lcm_lattice}  and \ref{linear_resolutions_cm_alexander_dual}, combined with Remark \ref{eagon_reiner_theorem_with_degrees}
which explains the degrees and dimensions of $L(I_\Delta)$ and $\Delta$.   The second part of this theorem stems from Theorems \ref{linear_quotients_cl_shellability} and \ref{linear_quotients_alexander_dual}, combined with Remark \ref{linear_quotients_alexander_dual_with_degrees} that explains the degrees and dimensions of $L(I_\Delta)$ and $\Delta$. 

\end{proof}
\begin{remark}
Corollary \ref{alexander_dual_lcm_lattice} breaks some symmetry between the lcm-lattice and the Alexander dual. Being pure/degree graded and Cohen-Macaulay happens in one object if and only if it happens in the other. On the contrary, regarding shellability, we see that shellability of the Alexander dual corresponds to CL-shellability in the lcm-lattice, which is a different (stronger) property than shellability of the lcm-lattice in general. An example of a shellable but not CL-shellable poset is given in \cite{Vince1985ASP}. See further discussion in \ref{further_research}.
\end{remark}

\section{The lcm-lattice of edge ideals}

In this section we discuss the implications on an interesting subclass of monomial ideals inherently connected to simple graphs.

An interesting case of monomial ideals is the case of \textbf{edge ideals}. For a simple graph $G=(V,E)$, its \textbf{edge ideal} is the ideal $$I(G)=(x^e,e\in G)$$
where $x^e=x_ix_j$ for $e=\{i,j\}$. 
An \textbf{induced subgraph} of a graph $G=(V,E)$ is a subgraph $H=(V^{\prime},E^{\prime})$ such that $V^{\prime}\subseteq V$ and 
$e\in E' \subseteq \binom{V'}{2}$
if and only if $e\in E$. The \textbf{complement} of a graph $G=(V,E)$ is the graph $\bar{G} = (V,\bar{E})$ where $e\in \bar{E}$ if and only if $e\notin E$.
A graph is called \textbf{chordal} if it has no induced cycles of length $\geq 4$.
Fr\"oberg \cite{RalfFröberg1990} gave a full combinatorial characterization of graphs whose edge ideals have linear resolutions:
\begin{theorem}
\label{froberg_theorem}
Let $G$ be a simple graph. $I(G)$ has a linear resolution if and only if $\bar{G}$ is chordal.
\end{theorem}
Herzog, Hibi and Zheng  showed the following:

\begin{theorem}[\cite{herzog2003monomialidealspowerslinear}, Theorem 3.2]
\label{linear_resolutions_equivalent_to_quotients}
    Let $I=I(G)$ be an edge ideal. The following are equivalent:
    \begin{enumerate}
        \item $I$ has linear resolution.
        \item $I$ has linear quotients.
    \end{enumerate}
\end{theorem}
Applying our result from Section~\ref{linear_quotients_ideals_cl_shellability_section} we prove Corollary \ref{chordality_cl_shellability}:

\begin{proof}
By Theorems \ref{froberg_theorem}+\ref{linear_resolutions_equivalent_to_quotients}, $$\bar{G} \text{ is chordal } {\iff} I \text{ has linear quotients }$$ And by Theorem \ref{linear_quotients_cl_shellability}, $$I \text{ has linear quotients }{\iff} L \text{ is CL-shellable } 2\text{-degree graded}$$
\end{proof}

Nevo \cite{nevo2009regularitytopologylcmlatticec4free} studied regularity of edge ideals via the lcm-lattice. One of his results was the following:
\begin{theorem}\label{thm:nevo}[\cite{nevo2009regularitytopologylcmlatticec4free}, Theorem 3.1]
    Let $G$ be a simple graph, $\bar{G}$  is chordal, $L=L(I(G))$. Then $O(\hat{L})$ is constructible.
\end{theorem}
A mistake was found in the proof and discussed with Nevo. Specifically, the claim at the end of the proof that $$\Delta_{l-1}\cap ((x_{v_l},m_G))\cong \Delta_{l-1}(G[V-v_l])$$ is incorrect. As a counter example, consider the graph $G'$ on vertices $\{x,y,z,w,u\}$ presented in Figure~\ref{fig:eran_counterexample_graph}.

\begin{figure}[ht]
    \centering
    \includegraphics[width=0.25\textwidth]{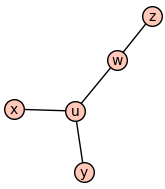}
    \caption{Illustration of the graph $G'$}
    \label{fig:eran_counterexample_graph}
\end{figure}

The ordering $x,y,z,w,u$ gives a Dirac ordering for the graph. Considering the lcm-lattice of the edge ideal (see Figure~\ref{fig:eran_counterexample_lcm_lattice}), and taking $l=5$, we see that $\text{dim}(\Delta_{l-1}\cap ((x_{v_l},m_{G'})))=1$ while $\text{dim}(\Delta_{l-1}(G'[V-v_l]))=-1$, hence such an isomorphism does not exist. 

\begin{figure}[ht]
    \centering
    \includegraphics[width=1.0\textwidth]{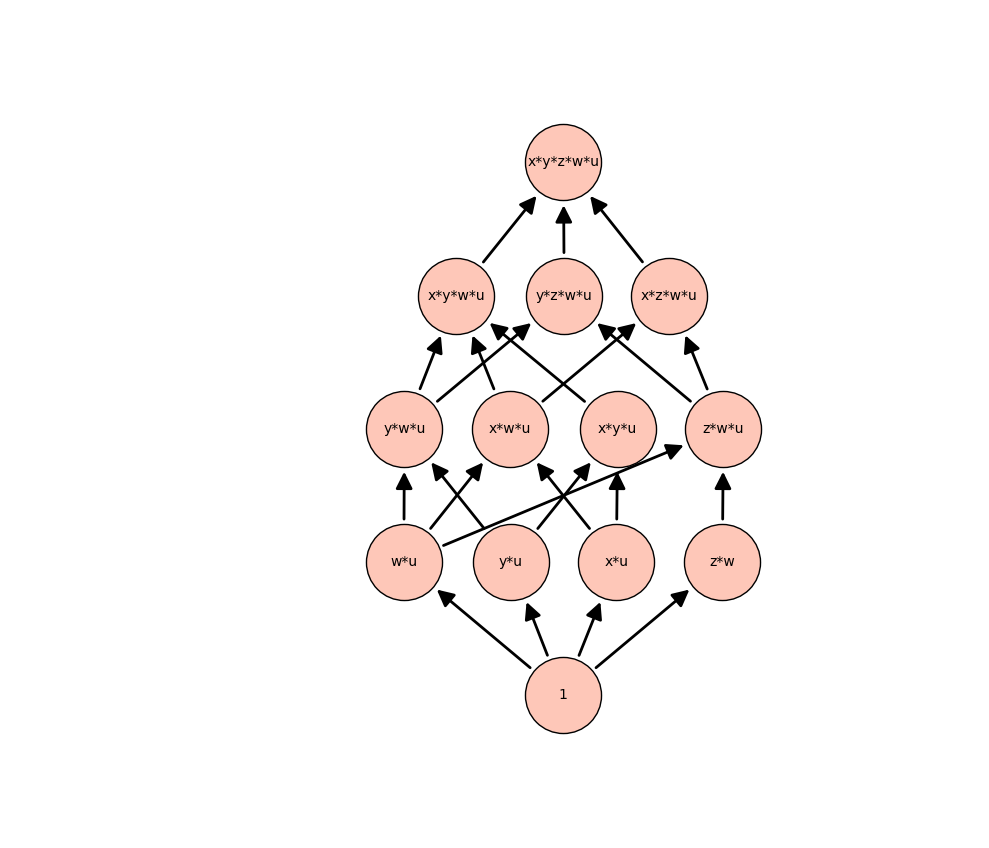}
    \caption{Illustration of the lcm-lattice $L(I(G'))$}
    \label{fig:eran_counterexample_lcm_lattice}
\end{figure}

However this is remedied by Corollary \ref{chordality_cl_shellability} since it is known that $\text{CL-shellability}\Rightarrow \text{shellability}\Rightarrow \text{constructibility}$, hence the assertion in Theorem \ref{thm:nevo} does hold.

\section{Further Research}
\label{further_research}

In this paper we characterized linear quotients and linear resolutions of monomial ideals in terms of their lcm-lattices. We showed that a monomial ideal has a $d$-linear resolution if and only if its lcm-lattice is $d$-degree graded and Cohen-Macaulay, and that a monomial ideal generated by monomials of degree $d$ has linear quotients if and only if its lcm-lattice is $d$-degree graded and CL-shellable. The lcm-lattice may have combinatorial properties that are stronger than being Cohen-Macaulay, but weaker than being CL-shellable. Notable ones are being constructible or shellable. A natural question that arises is:

\begin{question}
    Let $I$ be a monomial ideal, $L=L(I)$ its lcm-lattice. What algebraic property of $I$ is equivalent to $L$ being constructible? What algebraic property of $I$ is equivalent to $L$ being shellable?
\end{question}

Another interesting question that arises is whether in the case of lcm-lattices (equivalently, finite atomic lattices), shellability is equivalent to CL-shellability. It is known (\cite{bjorner1983lexicographically}, Theorem 3.3) that a graded poset that is CL-shellable, is shellable. Conversely, in \cite{Vince1985ASP} a poset that is shellable but not CL-shellable was given. Specifically, Vince and Wachs presented the graph $G$ given in Figure \ref{fig:vince_wachs_counter}.

\begin{figure}
    \centering
    \includegraphics[width=0.9\linewidth]{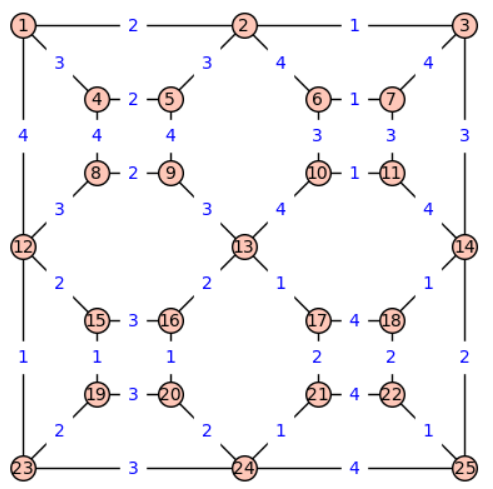}
    \caption{Vince-Wachs edge-colored graph $G$}
    \label{fig:vince_wachs_counter}
\end{figure}

They defined the poset $P(G)$ as follows: Let $I=\{1,2,3,4\}$ be the color set, $J\subseteq I$. $S_J$ is defined to be the set of connected components in the subgraph of $G$ consisting of edges with colors in $J$. $P(G)$ consists of pairs $(H,J)$ where $J\subseteq I$ and $H\in S_J$, and a minimal element $\hat{0}$. The order relation is defined by $(H,J)\leq (H', J')$ iff $H\subseteq H'$ and $J\subseteq J'$. Notice the maximal element $\hat{1}=(V(G),I)$. Theorem 3 in \cite{Vince1985ASP} shows that $P(G)$ is not CL-shellable, despite being shellable.

We will now show that $P(G)$ is not a lattice: consider the elements $x=(\{22\}, \emptyset)$, $y=(\{4\}, \emptyset)$. The upper bounds of both $x$ and $y$ in $P(G)$ are $(\{1,4,8,12,21,22,23,24,25\},\{1,3,4\})$, $(\{1,2,3,4,5,14,18,22,25\}, \{1,2,3\})$ and the maximal element $\hat{1}$. One can see that none of these is a least upper bound, hence these two elements don't have a join.

To our knowledge, there is no known example of a finite atomic lattice that is shellable but not CL-shellable. We thus pose the following question:

\begin{question}
\label{cl_shellability_and_shellability_question}
    Do there exist finite atomic lattices that are shellable but not CL-shellable?
\end{question}

\newpage
    \section*{Acknowledgements}
I would like to thank my thesis advisor, Eran Nevo, for his constant support and for useful discussions and advice. It has been a pleasure working with him, and a great learning opportunity.
    \bibliography{references}
    \bibliographystyle{plain}

\newpage
\appendix

\section{Appendix: monomial ideals with linear resolutions and their lcm-lattices}
\label{appendix:A}

In this appendix we will give an explicit proof of Theorem \ref{linear_resolutions_cm_lcm_lattice}.

The following theorem is a slight reformulation of (\cite{phan2005minimal}, Theorem 5.8):
\begin{theorem}
    \label{cm_lcm_lattice_equivalence_phan}
    Let $L$  be a finite atomic lattice, identified with $L(M(L))$ in the natural way. Then $M(L)$  has a $d$-linear resolution if and only if $L$  is Cohen-Macaulay as a poset and is $d$-degree graded.
\end{theorem}
We will now strengthen the formulation of Theorem \ref{cm_lcm_lattice_equivalence_phan}
  for monomial ideals in general, and not only for minimal monomial ideals for a given lattice. Notice that Phan proved that if an ideal $I$ has a linear resolution, then so does $M(L)$ where $L=L(I)$ is the lcm-lattice of $I$ (\cite{phan2005minimal}, Theorem 5.2), but not the other way around. We will rely heavily on the work of Phan \cite{phan2005minimal} and use some of his results and proofs but will not use his construction of minimal monomial ideals.

 \begin{definition}
    Let $I$ be a monomial ideal. The \textbf{$i-th$ Betti degrees} of $I$ are the monomials $m$ for which $\beta_{i,m}(I)\neq 0$. The first Betti degrees are the monomials $m$ for which $\beta_{1,m}(I)\neq 0$.
\end{definition}

\begin{definition}
    Let $I$ be a monomial ideal, $L=L(I)$ its lcm-lattice. A monomial $m$ is called a \textbf{super atom} of $L$ if it covers an atom of $L$.
\end{definition}

\begin{lemma}
\label{super_atoms_in_linear_resolutions}
    Let $I$ be a monomial ideal with a linear resolution, $L=L(I)$ its lcm-lattice. Then a monomial $m$ is a super atom of $L$ if and only if it is a first Betti degree of $I$. In particular, super atoms of $L$ are all of degree $d+1$.
\end{lemma}

\begin{proof}
    This stems from Theorem \ref{betti_lcm_lattice_connection}. If $m$ is a super atom, then $$\beta_{1,m}(I)=\textrm{dim} \ \tilde{H}_{0}(O(1,m);k)\neq 0.$$
    Indeed, by the definition of an lcm-lattice a super atom must be the lcm of at least two atoms, hence its interval has at least two connected components. Thus since $I$ has a $d$-linear resolution, we must have $\textrm{deg}(m)=d+1$.

    Conversely, assume $m$ is a first Betti degree of $I$. Then as discussed above it must have degree $d+1$ and it must belong to $L$ by Theorem \ref{betti_lcm_lattice_connection}. Thus by degree considerations it must be a super atom.
\end{proof}

\begin{lemma}[\cite{gasharov1999lcm}, Proposition 2.3(b)]
\label{correspondence between lattices}
    Let $\Delta$ be a simplicial complex. Consider the lattice $L_{\Delta ^{\vee}}$ whose elements are intersections of facets of $\Delta^{\vee}$, ordered by reverse inclusions (with the minimal element being the entire vertex set). Then $L(I_{\Delta})\cong L_{\Delta^{\vee}}$
\end{lemma}

\begin{proof}
    (\textit{sketch}) The correspondence between atoms and facets of the Alexander dual $m \leftrightarrow [n]\setminus \textrm{supp}(m)$ extends to a correspondence between the lcms and the appropriate intersections.
\end{proof}

\begin{lemma}[\cite{phan2005minimal}, Corollary 5.4]
\label{ideal_has_linear_resolution_then_I_1_as_well}
    Let $I$ be a squarefree monomial ideal with a $d$-linear resolution. Then $I_1$, which is the ideal generated by first Betti degrees of $I$ also has a linear resolution.
\end{lemma}

\begin{proof}
    We will first prove a dual statement about simplicial complexes. Let $\Delta$ be a simplicial complex. Denote $\Delta_1$ to be the simplicial subcomplex of $\Delta$ generated by intersections of at least two facets of $\Delta$. Let us show that $\Delta_1$ is pure of codimension 1, and that if $\Delta$ is Cohen-Macaulay then so is $\Delta_1$. 
    
    The part about being a subcomplex of codimension 1 follows from Lemmas \ref{correspondence between lattices} and \ref{super_atoms_in_linear_resolutions}: Indeed, consider the lattice of intersections of facets of $\Delta$ defined in Lemma \ref{correspondence between lattices}, $L_{\Delta}$. Then super atoms of $L_{\Delta}$ are precisely the facets of $\Delta_1$. By the Eagon-Reiner theorem, $\Delta$ being Cohen-Macaulay implies that $I_{\Delta^{\vee}}$ has a $(n-\textrm{dim}(\Delta)-1)$-linear resolution where $n$  is the number of vertices. By Lemma \ref{super_atoms_in_linear_resolutions} we get that super atoms of $L_{I_{\Delta^{\vee}}}$ are of degree $(n-\dim(\Delta))$.  Thus by the isomorphism between lattices from Lemma \ref{correspondence between lattices} we get that the super atoms of $L_\Delta$  are of cardinality $\textrm{dim}(\Delta)$, hence $\Delta_1$ is of codimension 1.
    
    As for Cohen-Macaulayness, 
    For $\textrm{dim}(\Delta)=1,0$ this is easily verified. 
    Assume this is true for $\textrm{dim}(\Delta)=l-1$, and let $\Delta$ be of dimension $l$. If $F$ is a face of $\Delta$ that is not empty, then $\textrm{dim}(\textrm{link}_{\Delta}(F))<l$. $\textrm{link}_{\Delta}(F)$ is CM of dimension $<l$, so by induction hypothesis $(\textrm{link}_{\Delta}F)_1$ is CM as well. Since $(\textrm{link}_{\Delta}F)_1=\textrm{link}_{\Delta_1}F$ we are done with this case.

    Assume now $F=\emptyset$. We need to show that $\Delta_1$ has no nontrivial homology in dimension $<l-1$. We will show that in these dimensions $\tilde{H_i}(\Delta_1)=\tilde{H}_i(\Delta)$ and finish by assumption that $\Delta$ is CM. Indeed, $\Delta$ is obtained from $\Delta_1$ by attaching $l$-simplices along codimension 1 faces. If such a simplex's boundary is fully contained in $\Delta_1$, then its addition only adds an $l$ chain to the chain complex and doesn't affect homology in dimensions $<l-1$. Otherwise, after the attachment of this simplex then there is a deformation retract of this simplex $G$ onto $2^G\cap \Delta_1$, hence we haven't changed the homology at all. So $\Delta_1$ is indeed CM.

     Now, let $I$ be a squarefree monomial ideal with a linear resolution. Applying as above the Stanley-Reisner correspondence, denote $\Delta=\Delta(I)^{\vee}$. Then as discussed, by Lemma \ref{correspondence between lattices} we have $\Delta_1=\Delta(I_1)^{\vee}$. By Eagon-Reiner, $\Delta$ is CM, hence by the above $\Delta_1$ is as well. Applying Eagon-Reiner again, we get that $I_1$ has a linear resolution and we are done.
\end{proof}

The following is Proposition 5.7 in \cite{phan2005minimal}:

\begin{lemma}
\label{linear resolution implies degree graded}
    Let $I$ be a squarefree monomial ideal with a $d$-linear resolution. Then $L=L(I)$ is $d$-degree graded.
\end{lemma}

\begin{proof}
    By induction on the length of $L$ (i.e. $\textrm{dim}(O(L))$). If the length is 1, $I$ is principal and $L$ has two the elements 1, $m_L$ and we are done. If we know the result for length $n-1$, then for length $n$ it suffices to show that for superatoms $m$, $\textrm{deg}(m)=d+1$. This is because by Lemma \ref{super_atoms_in_linear_resolutions} the lattice whose atoms are the super atoms is $L(I_1)$, so induction hypothesis will finish the proof by Lemma \ref{ideal_has_linear_resolution_then_I_1_as_well}. But the claim is true for super atoms by Lemma \ref{super_atoms_in_linear_resolutions} so we are done.
\end{proof}
The following is a known topological fact we will need:
\begin{lemma}
\label{Join of CM is CM}
    Let $\Delta_1$, $\Delta_2$ be simplicial complexes, $\Delta_1 \ast \Delta_2$ their join, $k$ a field. If  $\tilde{H_i}(\Delta_1;k)=0 \ \textrm{for all } i < \mathrm{dim}(\Delta_1)$ and $\tilde{H_i}(\Delta_2;k)=0 \ \textrm{for all } i < \mathrm{dim}(\Delta_2)$ then $\tilde{H_i}(\Delta_1 \ast \Delta_2;k) = 0 \ \textrm{for all } i < \mathrm{dim}(\Delta_1 \ast \Delta_2)$ 
\end{lemma}

\begin{proof}
    Recall the K\"unneth formula for the homology of a join of topological spaces (\cite{EoMJoin}):
    \[
\widetilde{H}_{r+1}(X \ast Y) \cong 
\bigoplus_{i+j=r} \widetilde{H}_i(X) \otimes \widetilde{H}_j(Y)
\oplus
\bigoplus_{i+j=r-1} \mathrm{Tor}(\widetilde{H}_i(X), \widetilde{H}_j(Y)).
\]
Since we are working over a field, we get:
    \[
\widetilde{H}_{r+1}(X \ast Y;k) \cong 
\bigoplus_{i+j=r} \widetilde{H}_i(X;k) \otimes \widetilde{H}_j(Y;k)
\]

Assume that $\tilde{H}_{r+1}(\Delta_1 \ast \Delta_2;k)\neq 0$ for some $r+1<\mathrm{dim}(\Delta_1 \ast \Delta_2)$, this means that $\textrm{there are } i,j$ with $i+j=r$ such that both $\tilde{H_i}(\Delta_1)\neq 0$ and $\tilde{H_j}(\Delta_2)\neq 0$. But $$\textrm{dim}(\Delta_1 \ast \Delta_2) = \textrm{dim}(\Delta_1)+\textrm{dim}(\Delta_2)+1$$ 
hence $r<\textrm{dim}(\Delta_1)+\textrm{dim}(\Delta_2)$. This means that either $i<\textrm{dim}(\Delta_1)$ or $j<\textrm{dim}(\Delta_2)$, contradicting our assumption.

\end{proof}
We will now prove the general connection between monomial ideals with linear resolutions and their lcm-lattices. This extends \ref{cm_lcm_lattice_equivalence_phan}, which proves the statement for minimal monomial ideals.

Now we prove the squarefree case:

\begin{theorem}
    \label{linear_resolutions_cm_lcm_lattice_squarefree}
    Let $I$ be a squarefree monomial ideal. Then $I$ has a $d$-linear resolution if and only if $L(I)$ is $d$-degree graded and Cohen-Macaulay.
\end{theorem}

\begin{proof}
    Assume $I$ has a $d$-linear resolution. Then by \ref{linear resolution implies degree graded} we get that $L(I)$ is $d$-degree graded.
    Let $F\in O(L(I))$ be a face in the order complex of $L(I)$. Then $\textrm{link}_{O(L)}(F)$ is a join of intervals of the form $(m_1, m_2), [1,m_1), (m_1, m_L]$ for some arbitrary monomials $m_1,m_2\in L$ or the entire lattice itself $O(L(I))$. We need to show that $\tilde{H}_i(\textrm{link}_{O(L)}(F);k)=0 \textrm{for all } i<\textrm{dim(link}_{O(L)}F)$.

    By \ref{Join of CM is CM}, it suffices to show that all non top dimensional homologies vanish for each one of the intervals above.
    For intervals that are closed or half closed, this is trivial because they are cones hence acyclic. Assume an interval of the form $(m_1, m_2)$ where $m_1\neq 1$. By \ref{d_degree_graded_gives_total_semimodularity}, $[m_1, m_2]$ is semimodular, hence shellable. Thus $(m_1, m_2)$ is shellable as well with the same shelling order. Thus it is Cohen-Macaulay, and in particular has no nontrivial non-top dimensional homology.
    
    We are left with the case of an interval of the form $(1,m)$. Assume $\tilde{H_i}(O(1,m);k)\neq 0$ for some $i$. Then by $\ref{betti_lcm_lattice_connection}$ we get that $\beta_{{i+1},m}(I)\neq 0$. Since $I$ as a $d$-linear resolution, we must have $\textrm{deg}(m)-i-1 = d \iff i=\textrm{deg}(m)-d-1$. But since $L(I)$ is $d$-degree graded, we have $\textrm{dim}(O(1,m))=\textrm{deg}(m)-d-1$, hence $i=\textrm{dim}(O(1,m))$ as desired.

    Conversely, assume that $L(I)$ is $d$-degree graded and Cohen-Macaulay.  We need to show that if $\beta_{i,j}(I)\neq 0$ then $i\geq j-d$. By the equation $$\beta_{i,j}(I)=\sum_{m\in \text{mon}(S), \text{deg}(m)=j}\beta_{i,m}(I)$$ we need to show that $\beta_{i,m}(I)=0$ $\textrm{for all } m\in \text{mon}(S), i<\text{deg}(m)-d$.  By \ref{betti_lcm_lattice_connection}, it suffices to show this for $m\in L(I)$.

    Indeed, let $m\in L(I)$. Choose any chain $E=\left(m=l_0\prec l_1 \prec \dots \prec l_t = m_L\right)$, then we get $\text{link}_{O(L)}(E\cup \{1\})=O((1, m))$. Since $L$ is Cohen-Macaulay and $E\cup \{1\}\in O(L)$, we get that $$\tilde{H}_i(O((1,m));k)=0 \ \textrm{for all } i<dim(O((1,m)))$$Now, because $L$ is $d$-degree graded, we get that $dim(O((1,m)))=deg(m)-d-1$. Thus
    $$\tilde{H}_i(O((1,m));k)=0 \ \textrm{for all } i<deg(m)-d-1$$
    Applying \ref{betti_lcm_lattice_connection} and changing indexes we get 
    $$\beta_{i,m}(I)=0 \ \textrm{for all } i<deg(m)-d$$ as desired.
    
\end{proof}

Let us generalize Theorem \ref{linear_resolutions_cm_lcm_lattice_squarefree} to monomial ideals in general using two simple lemmas. The first is well known, for example as stated in \cite{Morales_2014}:

\begin{lemma}
\label{linear_resolutions_polarizations}
    Let $I$ be a monomial ideal. Then $I$ has a $d$-linear resolution if and only if its polarization has a $d$-linear resolution.
\end{lemma}

The next is also an easy lemma regarding polarizations and lcm-lattices:

\begin{lemma}
    \label{polarization_lattices_cohen_macaulay}
    Let $I$ be a monomial ideal, $J$  its polarization, $L(I)
    $ and $L(J)$ the corresponding lcm-lattices. We then have:
    \begin{enumerate}
        \item $L(I)$ is $d$-degree graded iff $L(J)$ is $d$-degree graded.
        \item $L(I)$ is Cohen-Macaulay iff $L(J)$ is Cohen-Macaulay
    \end{enumerate}
\end{lemma}
\begin{proof}
    Any isomorphism of posets between $L(I)$ and $L(J)$ gives equivalence 2. The isomorphism from Lemma \ref{polarization_lattices_iso} preserves degrees, ensuring equivalence 1.
\end{proof}
We can now prove Theorem \ref{linear_resolutions_cm_lcm_lattice}:
\begin{proof}
    This follows from Theorem \ref{linear_resolutions_cm_lcm_lattice_squarefree}, Lemma \ref{linear_resolutions_polarizations} and Lemma \ref{polarization_lattices_cohen_macaulay}:
    
\begin{tikzcd}
I\text{ has linear a $d$-linear resolution} \arrow[d, Leftrightarrow, "\ref{linear_resolutions_polarizations}"] & & L(I)\text{ is }\text{$d$-degree graded and Cohen-Macaulay} \arrow[d, Leftrightarrow, "\ref{polarization_lattices_cohen_macaulay}"] \\
J\text{ has linear a $d$-linear resolution} \arrow[rr, Leftrightarrow, "\ref{linear_resolutions_cm_lcm_lattice_squarefree}"] & & L(J)\text{ is }\text{$d$-degree graded and Cohen-Macaulay}
\end{tikzcd}

\end{proof}

\end{document}